\documentclass[reqno,11pt]{amsart}

\usepackage{mathrsfs}
\usepackage{color,hyperref}
\usepackage{amsmath}
\usepackage{amssymb}
\usepackage[top=1in, bottom=1.25in, left=1.25in, right=1.25in]{geometry}

\newtheorem{thm}{Theorem}[section]

\newtheorem{lemma}[thm]{Lemma}
\newtheorem{proposition}[thm]{Proposition}
\newtheorem{definitio}[thm]{Definition}
\newtheorem{conjecture}[thm]{Conjecture}
\newtheorem{theorem}[thm]{Theorem}
\newtheorem{remark}[thm]{Remark}
\newtheorem{corollary}[thm]{Corollary}

\long\def\oo#1{}

\usepackage[numbers]{natbib}

\begin{document}

\title{Lattice-ordered abelian groups finitely generated as semirings}
\author{V\'\i t\v ezslav Kala}
\address{Faculty of Mathematics and Physics, Charles University, Sokolovsk\' a 83,
18600 Praha~8, Czech Republic}
\address{Mathematisches Institut, Georg-August Universit\" at G\" ottingen, Bunsenstra\ss e 3-5, 37073 G\" ottingen, 
Germany}
\email{vita.kala@gmail.com}

\subjclass[2010]{Primary 06F20, 12K10; Secondary 06D35, 16Y60, 52B20}

\keywords{Lattice-ordered abelian group, MV-algebra, parasemifield, semiring, finitely generated, order-unit}
\thanks{The author was supported by ERC Starting Grant 258713.}

\begin{abstract}
A lattice-ordered group (an $\ell$-group) $G(\oplus, \vee, \wedge)$ can be naturally viewed as a semiring $G(\vee,\oplus)$.
We give a full classification of (abelian) $\ell$-groups which are finitely generated as semirings, by first showing that each such $\ell$-group has an order-unit so that we can use the results of Busaniche, Cabrer and Mundici \cite{BCM}. 
Then we carefully analyze their construction in our setting to obtain the classification in terms of certain $\ell$-groups associated to rooted trees (Theorem \ref{classify}).

This classification result has a number of important applications: for example it implies a classification of finitely generated ideal-simple (commutative) semirings $S(+, \cdot)$ with idempotent addition and 
provides important information concerning the structure of general finitely generated ideal-simple (commutative) semirings, useful in obtaining
further progress towards Conjecture \ref{main-conj} discussed in \cite{BHJK}, \cite{JKK}.
\end{abstract}

\maketitle

\section{Introduction}\label{sec:1}

Lattice-ordered groups (or $\ell$-groups for short) have long played an important role in algebra and related areas of mathematics.
Let us briefly mention the relation to functional analysis and logic via the correspondence with MV-algebras \cite{M}, \cite{NL}, or the fact that the theory of factorization and divisibility on a B\' ezout domain yields an abelian $\ell$-group. For this and further applications see eg. \cite{AF} or \cite{GH}; the connections to B\' ezout domains were recently studied in detail by Yang \cite{Yang}.

Recently, there have been several interesting results concerning unital $\ell$-groups. For example, 
Busaniche, Cabrer and Mundici \cite{BCM} classified finitely generated unital (abelian) $\ell$-groups $G$ using the combinatorial notion of a stellar sequence, 
which is a sequence $|\Delta_0|\supset |\Delta_1|\supset\dots$ of certain simplicial complexes in $[0,1]^n$. 
The idea is that each such $G$ is of the form $G\simeq\mathcal M([0, 1]^n)/I$, where 
$\mathcal M([0, 1]^n)$ is the $\ell$-group of all piecewise linear functions  $f:[0, 1]^n\rightarrow \mathbb R$ and $I$ is the set of all functions $f$ such that $f(|\Delta_i|)=0$ for some $i$.

The aim of this paper is to explore and use the connections between semirings and $\ell$-groups in the study of simple semirings.
Namely, an $\ell$-group $G(\oplus, \vee, \wedge)$ is also a semiring $G(\vee,\oplus)=S(+, \cdot)$ such that the semiring addition $+$ is idempotent. 
By removing the idempotency condition, one obtains the notion of a parasemifield, i.e., a commutative semiring $S(+, \cdot)$ such that its multiplicative structure forms a group. (See the beginnings of Sections \ref{sec:2} and \ref{sec:3} for precise definitions of the notions concerning $\ell$-groups and semirings, respectively.)

In fact, it is not hard to observe that there is a term-equivalence between lattice-ordered groups and additively idempotent parasemifields (i.e., satisfying $a+a=a$ for all $a$). In particular, this equivalence preserves finite generation in the sense that an $\ell$-group is finitely generated if and only if it is finitely generated as a parasemifield. However, these are not equivalent to the property of being finitely generated as a semiring, which is stronger.

We shall assume all $\ell$-groups and semirings to be automatically commutative, as we will be dealing only with these throughout the paper.

Our first result is Theorem \ref{unital} in which we show that every additively idempotent parasemifield, finitely generated as a semiring, is unital in the $\ell$-group sense. Hence it is natural to inquire whether we can identify the ones, which are finitely generated as semirings, among the unital $\ell$-groups from the classification \cite{BCM}. 
The answer is yes, although the proof is fairly involved and requires a careful discussion of the geometry of stellar sequences. The resulting Theorem \ref{classify} classifies all additively idempotent parasemifields which are finitely generated as semirings.

\

While this seems to be the first paper to systematically study semirings from the perspective of $\ell$-groups and MV-algebras and to apply the strong classification results available therein to semirings, 
there is a long and fruitful tradition of proceeding in the other way, namely, of attaching a semiring to an MV-algebra (note that MV-algebras are equivalent to unital $\ell$-groups via the Mundici functor \cite{M}). 
This was started by Di Nola and Gerla \cite{DNG} who defined an MV-semi\-ring attached to an MV-algebra. 
Belluce and Di Nola \cite{BDN-C} simplified it to an equivalent definition of MV-semirings.
These two authors and Ferraioli \cite{BDNF-B} then 
established a categorical equivalence between MV-algebras and MV-semirings and used it to obtain representation
of MV-algebras as certain spaces of continuous functions via a corresponding representation of MV-semirings. The same authors \cite{BDNF} then very recently continued in the study of (prime) ideals of MV-semirings.
Let us also note that there is a wealth of other interesting representation results for MV-algebras, see eg. \cite{BDN-F}, \cite{BDN-E}, or \cite{GRS}. 
Finally, we remark that Belluce and Di Nola \cite{BDN-C} have also established a connection to ring theory by studying a class of ``\L ukasiewicz rings", which are defined as the rings whose semirings of ideals form an MV-algebra.

Given the basic and fundamental nature of the notion of a semiring,  
it is not surprising that there is a wide variety of other applications of semirings and semifields, ranging from cryptography and other areas of computer science to dequantization, tropical mathematics and geometry -- 
see for example \cite{G}, \cite{litvinov}, \cite{monico}, and \cite{zumbragel} for overviews of some of the applications and for further references.

Many parts of the structural theory of semirings and semifields mimic analogous results concerning rings and fields, see, e.g., \cite{G}. 
However, much less is known overall: for instance, whereas simple commutative rings are just fields and are known very explicitly, the analogous results for semirings are more subtle.
First of all, one has to distinguish between congruence-simple and ideal-simple semirings.
Bashir, Hurt, Jan\v ca\v r\' \i k and Kepka \cite{BHJK} classified the congruence-simple ones and
reduced the study of ideal-simple semirings to the study of parasemifields. 

Together with their results, our Theorem \ref{classify} implies a full classification of additively idempotent finitely generated ideal-simple semirings. 
The structure of this classification follows Theorem 11.2 and Section 12 of \cite{BHJK}, but it is fairly technical, so we don't state the final result explicitly.

\

We have already mentioned that additively idempotent parasemifields are term-equivalent to $\ell$-groups;
the present Theorem \ref{classify} classifies those which are finitely generated semirings.
A natural question to ask then is what is the structure of such parasemifields without the idempotency assumption. Note that the corresponding result concerning rings is that if a field is finitely generated as a ring, then it is finite. 

There are no finite parasemifields and in fact, we have the following conjecture:

\begin{conjecture}[\cite{BHJK}, \cite{JKK}]\label{main-conj}
Every parasemifield which is finitely generated as a semiring is additively idempotent.
\end{conjecture}

Je\v zek, Kala and Kepka \cite{JKK} proved this in the case of at most two generators by studying the geometry of semigroups $\mathcal C(S)\subset\mathbb N_0^2$ attached to parasemifields $S$. 
(For the definition and basic information on the semigroups $\mathcal C(S)$, see Section \ref{order-unit}.) 
Since each parasemifield $S$ has an additively idempotent factor $S/\sim$ such that the semigroup $\mathcal C(S)$ is equal to $\mathcal C(S/\sim)$, 
one can use Theorem \ref{classify} to obtain refined information on the structure of the semigroup $\mathcal C(S)\subset\mathbb N_0^m$ in general.

In a work in progress \cite{KaKo}, the author and Korbel\' a\v r use this to prove Conjecture \ref{main-conj} in the case of three generators. 
It seems quite possible that a similar approach will yield a proof of this conjecture in general.
Our Theorem \ref{classify} would then provide all parasemifields, finitely generated as a semiring and hence, again using the results of \cite{BHJK}, imply a complete classification of finitely generated ideal-simple semirings (see \cite{KK} for some details and page \pageref{example} of the present paper for an example).

\

There are various natural ways of extending and generalizing the classification of finitely generated unital $\ell$-groups \cite{BCM}. Let us just mention the cases of $\ell$-groups which are not assumed to be unital, of finitely generated parasemifields, or even of non-commutative finitely generated parasemifields. To the author's knowledge, not much is known about any of these interesting problems.

\

As for the contents of this paper, Section \ref{sec:2} reviews the definitions and basic facts on $\ell$-groups, including the statement of the classification of finitely generated unital ones (and the required notions concerning simplicial and abstract complexes). 
Then in Section \ref{sec:3} we briefly review some preliminaries on semirings and parasemifields and prove that if an additively idempotent parasemifield is finitely generated as a semiring, then it is unital. 
For the sake of completeness we outline the proofs of some classical results concerning semirings that we need.
In Section \ref{sec:4} we then give the classification of additively idempotent parasemifields, finitely generated as semirings.

\section*{Acknowledgments}

The results in this paper form a part of my PhD thesis \cite{disertace} 
written at Charles University, Prague, under the guidance of Professor Tom\' a\v s Kepka. I want to thank him for his help and for our enjoyable research collaboration.
I would also like to greatly thank the anonymous referee for a thorough reading of the manuscript and for a number of useful and detailed remarks and suggestions.

\section{$\ell$-groups and complexes}\label{sec:2}

In this section we briefly review some basics about $\ell$-groups and simplicial complexes that we will need, including the classification of Busaniche, Cabrer and Mundici \cite{BCM}. 
Our treatment is quite terse, but we at least try to provide a brief (very) informal overview at the end of this section. For a more detailed treatment we refer the reader to the paper \cite{BCM}. Also see \cite{CM}, where rational polyhedra are used in the study of projective unital $\ell$-groups.
For more general background information on $\ell$-groups see for example \cite{AF} or \cite{GH}.

A \emph{lattice-ordered abelian group} (\emph{$\ell$-group} for short) $G(+, -, 0, \vee, \wedge)$ is an algebraic structure such that $G(+, -, 0)$ is an abelian group, 
$G(\vee, \wedge)$ is a lattice, and $a+(b\vee c)=(a+b)\vee(a+c)$ for all $a, b, c\in G$. 

An \emph{order-unit} $u\in G$ is an element such that for each $g\in G$ there exists $n\in\mathbb N$ so that $nu\geq g$ (i.e., $nu\vee g=nu$). A \emph{unital $\ell$-group} $(G, u)$ is an $\ell$-group with an order-unit $u$. A \emph{unital $\ell$-homomorphism} is a homomorphism of $\ell$-groups which maps one order-unit to the other one.
An \emph{$\ell$-ideal}  is the kernel of a unital $\ell$-homomorphism; any $\ell$-ideal $I$ then determines the factor-homomorphism $G\rightarrow G/I$.

Let us now review the classification of \cite{BCM}. Denote by $\mathcal M([0, 1]^n)$ the set of piecewise linear continuous functions $f: [0, 1]^n\rightarrow\mathbb R$ such that each piece has integral coefficients (and the number of pieces is finite). $\mathcal M([0, 1]^n)$ is a group under pointwise addition of functions and
we can define $(f\vee g)(x)=\max(f(x), g(x))$ and $(f\wedge g)(x)=\min(f(x), g(x))$ for $f, g\in\mathcal M([0, 1]^n)$. This makes $\mathcal M([0, 1]^n)$ an $\ell$-group with the constant function 1 being an order-unit. Notice that $\mathcal M([0, 1]^n)$ is (finitely) generated (as an $\ell$-group) by the constant function 1 and the projections on $i$-th coordinate $\pi_i: [0,1]^n\rightarrow\mathbb R$ (but it is not finitely generated as a semiring, as we shall see in Corollary \ref{simplex}). Also, for $D\subset [0, 1]^n$ we define $\mathcal M(D)$ as the $\ell$-group whose elements are restrictions $f|D$ of functions $f\in\mathcal M([0, 1]^n)$ to $D$. $\mathcal M(D)$ is thus a factor of $\mathcal M([0, 1]^n)$.
 
The classification then says that each finitely generated unital $\ell$-group is of the form $\mathcal M([0, 1]^n)/I$ for an explicitly defined $\ell$-ideal $I$
(and provides a criterion for when two ideals give the same $\ell$-group). The ideal $I$ comes from a stellar sequence $\mathcal W$ of simplicial complexes as follows: from $\mathcal W$ we construct a sequence $\mathcal P_0\supset \mathcal P_1\supset \mathcal P_2\supset \dots$ of polyhedra in $[0, 1]^n$ and define $I=\{f\in\mathcal M([0, 1]^n)|f(P_i)=0$ for some $i\}$. To give more details we first need to give some definitions concerning (abstract) simplicial complexes, following \cite{BCM}.

We assume the reader is familiar with the usual notion of a (simplicial) complex in $\mathbb R^n$. Let us just note that a \emph{simplex} is a convex hull of a finite set of points, a \emph{$k$-simplex} is a simplex of dimension $k$, a \emph{complex} $\mathcal K$ is a finite set of simplexes such that if $T_1, T_2$ are simplexes with 
$\dim T_1=\dim T_2-1$, $T_1\subset\partial T_2$, and $T_2\in\mathcal K$, then also $T_1\in\mathcal K$ (where by $\partial T$ we denote the boundary of $T$). The \emph{support} $|\mathcal K|$ of a complex $\mathcal K$ is the union of all simplexes in $\mathcal K$. Throughout this paper
we shall often identify a complex with its support.
A simplex $\mathrm{conv}(v_0, \dots, v_k)$ is \emph{rational} if all the coordinates of all the vertices $v_i$ are rational.  A complex is \emph{rational} if all its simplexes are rational. For more background information on simplicial complexes, see for example \cite{Ewald}.

\begin{definition}[\cite{BCM}, page 262]
A (finite) \emph{abstract simplicial complex} is a pair $H=(\mathcal V, \Sigma)$,
where
$\mathcal V$ is a non-empty finite set of vertices of $H$ and $\Sigma$ is a
collection of subsets of $\mathcal V$ whose union is $\mathcal V$ with the property that every subset
of an element of $\Sigma$ is again an element of $\Sigma$. 
Given $\{v, w\}\in\Sigma$ and $a\not\in\mathcal V$ we define the \emph{binary subdivision} $(\{v, w\}, a)$ of $H$ as the abstract simplicial complex $(\{v, w\}, a)H$ obtained by
adding $a$ to the vertex set and replacing every set 
$\{v,w, u_1,\dots, u_t\}\in\Sigma$ by the
two sets $\{v,a, u_1,\dots, u_t\}$ and $\{a,w, u_1,\dots, u_t\}$ and all their subsets. 

A \emph{weighted
abstract simplicial complex} is a triple $W=(\mathcal V, \Sigma, \omega)$ where $(\mathcal V, \Sigma)$ is an abstract
simplicial complex and $\omega$ is a map of $\mathcal V$ into $\mathbb N$. For $\{v, w\}\in\Sigma$ and $a\not\in\mathcal V$, the \emph{binary subdivision} $(\{v, w\}, a)W$ is the abstract simplicial complex
$(\{v, w\}, a)(\mathcal V, \Sigma)$ equipped with the weight function 
$\tilde\omega: \mathcal V\cup\{a\}\rightarrow\mathbb N$
given by $\tilde\omega(a)=\tilde\omega(v)+\tilde\omega(w)$ and $\tilde\omega(u)=\omega(u)$ for all $u\in\mathcal V$.
\end{definition}

\begin{definition}[\cite{BCM}, page 264]
Let $W=(\mathcal V, \Sigma, \omega)$ and $W^\prime$ be two weighted abstract simplicial complexes. A map $\mathfrak b: W\rightarrow W^\prime$ is a \emph{stellar transformation} if it is either a deletion of a maximal set of $\Sigma$ or a binary subdivision or the identity map.

A sequence $\mathcal W=(W_0, W_1, W_2, \dots)$ of  weighted abstract simplicial complexes is \emph{stellar} if $W_{i+1}$ is obtained from $W_i$ by a stellar transformation.
\end{definition}

\begin{definition}[\cite{BCM}, page 263]
Let now $W=(\mathcal V, \Sigma, \omega)$ be an abstract simplicial complex with the set of vertices $V=\{v_1, \dots, v_n\}$. 
Choose the standard basis $e_1, \dots, e_n$ of $\mathbb R^n$ 
and let $\Delta_W$ be the complex whose vertices are 
$v_1^\prime=e_1/\omega(v_1), \dots,$ $v_n^\prime=e_1/\omega(v_n)$ and whose $k$-dimensional simplexes are given by 
$\mathrm{conv}(v_{i(0)}^\prime, \dots, v_{i(k)}^\prime)$ $\in\Delta_W$ if and only if $\{v_{i(0)}, \dots, v_{i(k)}\}\in\Sigma$.

Then $\Delta_W$ is a complex, $|\Delta_W|\subset [0,1]^n$ and we have a 
map $\iota: \mathcal V\rightarrow |\Delta_W|$ 
given by $\iota(v_i)=v_i^\prime$, the so called \emph{canonical realization} of $W$.
\end{definition}

\begin{definition}[\cite{BCM}, pages 256-257]
Let $\mathcal K$ be a complex and $p\in|\mathcal K|\subset \mathbb R^n$ a point in $\mathcal K$. The \emph{blow-up} $\mathcal K_{(p)}$ of $\mathcal K$ at $p$ is 
the complex obtained by replacing each simplex $T\in\mathcal K$ that contains $p$ by the set of all simplexes of the form $\mathrm{conv}(F\cup\{p\})$, where $F$ is any face of $T$ 
not containing $p$.

For a rational 1-simplex $E=\mathrm{conv}(v, w)\in\mathbb R^n$ we define the \emph{Farey mediant} of $E$ as the rational point
 $u=\frac{\mathrm{den}(v)v+\mathrm{den}(w)w}{\mathrm{den}(v)+\mathrm{den}(w)}\in E$
(where $\mathrm{den}(v)$ denotes the least common denominator of the coordinates of a vector $v$).

If $E$ belongs to a rational complex $\mathcal K$ and $v$ is the Farey mediant of $E$, the (binary) \emph{Farey blow-up} is the blow-up $\mathcal K_{(v)}$.
\end{definition}

\begin{remark}[\cite{BCM}, Lemma 4.4]\label{4.2.5}
Note that if $\mathcal W=(W_0, W_1, W_2, \dots)$ is a stellar sequence of weighted abstract simplicial complexes and $\iota_0: \mathcal V_0\rightarrow |\Delta_0|$ the canonical realization, we can naturally extend this to attach a complex $\Delta_i=\Delta_{W_i}$ to each $W_i$:

Let $\mathfrak b_0: W_0\rightarrow W_1$ be the given stellar transformation. We define $\Delta_1$ as follows: If $\mathfrak b_1$ deletes a maximal set $M\in\Sigma$, we delete the corresponding maximal simplex from $\Delta_{0}$. 
If $\mathfrak b_1$ is a binary subdivision $(\{a, b\},c)W_0$ at some $E=\{a,b\}\in\Sigma$, let $e$ be the Farey mediant of the 1-simplex 
$\mathrm{conv}(\iota_0(E))$. Then  $\Delta_1$ is the Farey blow-up of $\Delta_0$ at $e$.
If $\mathfrak b_1$ does not do anything, we also keep $\Delta_{0}$ unchanged.

In all cases we accordingly modify $\iota_0$ to obtain a realization $\iota_1: \mathcal V_1\rightarrow |\Delta_{1}|$.
Then we can continue by considering $\mathfrak b_1: W_1\rightarrow W_2$, and so on.

Eventually we get a sequence of complexes corresponding to $[0,1]^n\supset |\Delta_0|\supset |\Delta_1|\supset \dots$.
\end{remark}

\begin{definition}[\cite{BCM}, Lemma 2.3]
Given a sequence $\mathcal P=(P_1\supset P_2\supset\dots)$ of subsets of $[0,1]^n$, define an $\ell$-ideal $I=I(\mathcal P)$ of $\mathcal M([0,1]^n)$ 
by $I(\mathcal P)=\{f\in\mathcal M([0,1]^n)|f(P_i)=0$ for some $i\}$. This gives an $\ell$-group $\mathcal M([0,1]^n)/I(\mathcal P)$.
\end{definition}

\begin{theorem}[\cite{BCM}, Theorem 5.1]
For every finitely generated unital $\ell$-group $(G,u)$ there is a stellar sequence $\mathcal W=(W_0, W_1, W_2, \dots)$ such that 
$(G,u)\simeq \mathcal G(\mathcal W)$, where $\mathcal G(\mathcal W)=\mathcal M([0,1]^n)/I$ for $I=$ the ideal corresponding to the sequence
$[0,1]^n\supset |\Delta_0|\supset |\Delta_1|\supset \dots$ defined using $\mathcal W$ as in \ref{4.2.5}.
\end{theorem}

All this is not nearly as complicated as it sounds: we start with suitable complex $\Delta_0$  
and then modify it in infinitely many steps. In each step we either 

\begin{itemize}
\item delete a maximal simplex from the previous complex, or 

\item suitably divide a 1-dimensional simplex $E$ into two (and then we have to correspondingly divide all the simplexes containing $E$), or

\item don't do anything.
\end{itemize}

This produces a sequence $[0,1]^n\supset |\Delta_0|\supset |\Delta_1|\supset\dots$ and we define $G=\mathcal M([0, 1]^n)/I$, where $I$ is the set of all functions $f$ such that $f(|\Delta_i|)=0$ for some $i$. Every finitely generated unital $\ell$-group is obtained in this way.

\section{Existence of order-unit}\label{order-unit}\label{sec:3}

Let us now review the connection between $\ell$-groups and semirings.

\nopagebreak

By a (commutative) \emph{semiring} we shall mean a non-empty set $S$ equipped with two associative and commutative operations (addition and multiplication) where 
the multiplication distributes over the addition from both sides. We shall be dealing with commutative semirings only, so we usually just call them semirings. Note that our definition of a semiring is slightly more general from the one used in the context of MV-semirings (see eg. \cite{BDNF-B}) in that we don't require a semiring to contain 0 or 1.

A non-trivial semiring $S$ is a \emph{parasemifield} if the multiplication
defines a non-trivial group. A non-trivial semiring $S$ is a \emph{semifield} if there is an element $0\in S$ such that $0\cdot S=0$ and such that the set $S\backslash\{0\}$ is a group (for the semiring multiplication).

A semiring is \emph{additively idempotent} if $x+x=x$ for all $x\in S$.

As we mentioned already in the introduction, there is a well-known term-equiva\-len\-ce (and hence a categorical isomorphism) between additively idempotent parase\-mifields and $\ell$-groups. We shall use this to switch between the languages of parase\-mifields and $\ell$-groups, sometimes without explicitly mentioning it.

\begin{proposition}[\cite{WW}, \cite{W}]\label{3.1}
There is a term-equivalence between additively idempotent parasemifields and $\ell$-groups.
\end{proposition}

\begin{proof}
Let $S(+, \cdot, ^{-1}, 1)$ be an additively idempotent parasemifield and define $a\vee b=a+b$, $a\wedge b=(a^{-1}+b^{-1})^{-1}$. Then $S(\cdot, ^{-1}, 1, \vee, \wedge)$ is an $\ell$-group.
Conversely, if $S(\cdot, ^{-1}, 1, \vee, \wedge)$ is an $\ell$-group (written multiplicatively), then $S(+, \cdot, ^{-1}, 1)$ is an additively idempotent parasemifield, where $a+b=a\vee b$. 
We see that every basic operation on an $\ell$-group is a term operation on an additively idempotent parasemifield and vice versa. This implies that these two classes of algebras have the same clones of operations, i.e., they are term-equivalent.
\end{proof}

We define a (pre-)ordering $\leq$ on a semiring $S$ by $a\leq b$ if and only if $a=b$ or there exists $c\in S$ such that $a+c=b$. Note that it is preserved by addition and multiplication in $S$. Also, this is the same ordering as the one on the corresponding $\ell$-group. 

Note that if $S$ is a parasemifield, then the ordering $\leq$ on $S$ is antisymmetric:

\begin{proposition}[\cite{G}, Proposition 20.37]\label{P3}

Let $S$ be a parasemifield.
For all $a, b, c\in S$ we have:

a) If $a+b+c=a$, then $a+b=a$.

b) If $a\leq b\leq a$, then $a=b$.
\end{proposition}

\begin{proof}
a) Let $a+b+c=a$. Multiply both sides by $a^{-2}b$ and then add $a^{-1}c$. We get $a^{-1}b+a^{-2}b^{2}+a^{-2}bc+a^{-1}c=a^{-1}b+a^{-1}c$, and so 
$(a^{-1}b+a^{-1}c)(a^{-1}b+1)=(a^{-1}b+a^{-1}c)$. Dividing by $a^{-1}b+a^{-1}c$ we get $a^{-1}b+1=1$ as needed.

b) Write $b=a+x$ and $a=a+x+y$. By part a), $a=a+x=b$.
\end{proof}

\begin{definition}\label{Def2}
An additively idempotent parasemifield $S$ is \emph{order-unital} if there exists an element $u\in S$ such that for each $s\in S$ there is 
$n\in\mathbb N$ so that $u^ns+1=1$.
\end{definition}

\smallskip

Note that this definition is equivalent to the corresponding definition of a unital $\ell$-group. For if $v\in S$ is an order-unit in the $\ell$-group sense, we have 
that for each $s\in S$ there is some $n\in\mathbb N$ so that $v^n\geq s$. Choose now $u=v^{-1}$. Then $1\geq u^ns$, and so $1=u^ns+t$ for some $t\in S$. 
Now $1+u^ns=(u^ns+t)+u^ns=u^ns+t=1$. Conversely, if $u$ is an element from the definition \ref{Def2}, then $v=u^{-1}$ will be an order-unit in the $\ell$-group sense.

As usual, $\mathbb N$ and $\mathbb Q^+$ denote the semirings of positive integers and rational numbers, respectively;  $\mathbb N_0$ is the semiring of non-negative integers. 

\

While not as many classes of semirings have been studied as in the case of $\ell$-groups or MV-algebras, 
let us mention at least some examples of parasemifields and simple semirings in order to give our presentation a more concrete flavour. 
Basic example of additively idempotent parasemifields is given by a totally ordered group $G$ (written multiplicatively), where we define the semiring addition $a+b:=\mathrm{max}(a, b)$, obtaining so-called ``tropical semirings" or ``max-plus algebras". Standard examples are 
$\mathbb R(\max, +)$ and $\mathbb Z(\max, +)$. Note that the parasemifields we define in Definition \ref{G(T,v)} are a generalization of the latter case.

Simple semirings were considered in detail by Bashir, Hurt, Jan\v ca\v r\' \i k and Kepka \cite{BHJK}. The study of ideal-simple ones reduces to the case of parasemifields (not necessarily additively idempotent); a basic example of such a construction is the following: 
For a parasemifield $P(+, \cdot, ^{-1}, 1)$ consider the disjoint union $S:=P\cup\{0\}$ and extend the operations by setting $x+0=x$ and $x\cdot 0=0$. Then the semifield $S$ is an ideal-simple semiring.
\label{example}

Congruence-simple semirings are essentially completely classified, with the exception of a rather mysterious class of subsemirings $S$ of positive real numbers $\mathbb R^+(+, \cdot)$. 
We refer the interested reader to \cite{BHJK} for details and only note that the author and Korbel\' a\v r \cite{KaKo1} have provided examples of congruence-simple subsemirings of $\mathbb Q^+$ defined using $p$-adic valuations, such as 
$S=\{x\in\mathbb Q^+\mid 2^{-\mathrm{v}_p(x)}<x\}$ (here $\mathrm{v}_p(x)$ is the additive $p$-adic valuation of $x$).
 
For more background information on semirings see eg. \cite{G}.

\

We will need some basic properties of (finitely generated) parasemifields.

In the rest of this section, let $S$ be a parasemifield $m$-generated as a semiring. That means that there is a surjective semiring homomorphism $\varphi: \mathbb N[x_1, \dots, x_m]$ $\rightarrow S$ (where $x_i$ are indeterminates). For $a=(a_1, \dots, a_m)\in\mathbb N_0^m$ we use the notation 
$x^a=x_1^{a_1}\cdots x_m^{a_m}$.

Let $A$ be the prime subparasemifield of $S$, i.e., the smallest (possibly trivial) parasemifield contained in $S$.

Let $Q$ be the subsemiring of elements which are smaller than some element of $A$, i.e., $Q=\{s\in S|\exists q\in A: s\leq q\}$.  
Let $\mathcal C=\mathcal C(S)=\{a\in\mathbb N_0^m|\varphi(x^a)\in Q\}$ be the corresponding semigroup (or a cone) in $N_0^m$.

The structure of $Q$ and $\mathcal C$ carries a lot of information about $S$. For example, in \cite{JKK} it was used to show that every parasemifield,  two generated as a semiring, is additively idempotent.

\begin{proposition}\label{P1}
a) If $S$ is additively idempotent, then $A=\{1\}$. Otherwise $A\simeq\mathbb Q^+$.

b) For $q_1, q_2\in S$ we have $q_1+q_2\in Q$ if and only if $q_1,q_2\in Q$. For $q\in S$, $n\in\mathbb N$ we have $q^n\in Q$ if and only if $q\in Q$.

c) $\mathcal C$ is a pure subsemigroup of $\mathbb N_0^m$, i.e., it is closed under addition and for $n\in\mathbb N$ and $a\in\mathbb N_0^m$ we have $na\in\mathcal C$ if and only if $a\in\mathcal C$.
\end{proposition}

\begin{proof}
This a summary of various statements in \cite{JKK} and \cite{KKK}. We just sketch the proofs.

a) $\mathbb Q^+$ is the free 0-generated parasemifield. Therefore $A$ is a factor of $\mathbb Q^+$. Now it suffices only to note that $\mathbb Q^+$ is congruence simple.

b) $q_1+q_2\in Q$ means that $q_1+q_2\leq s$ for some $s\in A$. Therefore $q_i\leq q_1+q_2\leq s$ and $q_i\in Q$ ($i=1, 2$).

c) Follows directly from b).
\end{proof}

We shall use the structure of $\mathcal C$ to show Theorem \ref{unital}. In particular, we will need the following proposition which essentially says that there is an element $c$ which is ``inside" the cone $\mathcal C$.

\begin{proposition}\label{P2}
There exists $c\in\mathcal C$ such that:

\nopagebreak

a) $c+e_i\in\mathcal C$ for each $i=1, \dots, m$, where $e_i=(0, \dots, 0, 1, 0, \dots, 0)\in\mathbb N_0^m$ is the vector having 1 at the $i$-th position and 0 elsewhere.

b) $nc+a\in\mathcal C$ for each $a=(a_i)\in\mathbb N_0^m$, where $n=a_1+\dots+a_m$.
\end{proposition}

\begin{proof}
a) Take $f=1+x_1+\dots+x_m\in\mathbb N[x_1, \dots, x_m]$. Since $S$ is a parasemifield, there is $g=\sum_j a_j x^{c^{(j)}}$ (where $a_j\in\mathbb N$) such that $\varphi(g)$ is the inverse of $\varphi(f)$ in $S$, i.e., $\varphi(fg)=1$. Thus $\varphi(fg)\in Q$, and since 
$fg= \sum_j a_j (x^{c^{(j)}}+x^{c^{(j)}+e_1}+\dots+x^{c^{(j)}+e_m})$, by \ref{P1} b), each of the monomials $x^{c^{(j)}}, x^{c^{(j)}+e_1}, \dots, x^{c^{(j)}+e_m}$ lies in $Q$, and so $c^{(j)}, c^{(j)}+e_1, \dots, c^{(j)}+e_m$ all lie in $\mathcal C$. Hence we can just choose $c=c^{(j)}$ for any $j$.

b) $a=a_1e_1+\dots+a_me_m$, and so $nc+a=a_1(c+e_1)+\dots+a_m(c+e_m)\in \mathcal C$.
\end{proof}

We are now ready to prove the main result of this section:

\begin{theorem}\label{unital}
Let $S$ be an additively idempotent parasemifield, finitely generated as a semiring. Then $S$ is order-unital.
\end{theorem}

\begin{proof}
Choose $u=\varphi(x^c)$ with $c\in\mathcal C$ chosen by Proposition \ref{P2}. We want to show that $u^ns+1=1$ for each $s\in S$ and some $n\in\mathbb N$. Clearly it suffices to show it for $s=\varphi(x^a), a\in\mathbb N_0^m$ (each element of $S$ is a finite sum of elements of this form).

By \ref{P2} b), we can choose $n$ large enough so that $nc+a\in\mathcal C$. Thus $u^ns\in Q$.
Since $S$ is additively idempotent, $A=\{1\}$, and so this means that $u^ns\leq 1$.
Therefore $1\leq u^ns+1\leq 1$, and so by Proposition \ref{P3}, $u^ns+1=1$ and $S$ is unital.
\end{proof}

Note that the order-unit we have just constructed is in no way unique. However, we shall see that the resulting classification \ref{classify} is independent of the choice of the order-unit.

\

\section{The classification}\label{sec:4}

By Theorem \ref{unital} we know that every additively idempotent parasemifield, finitely generated as a semiring, is order-unital. 
Considering it as an $\ell$-group via Proposition \ref{3.1}, we see that it is one of the $\ell$-groups classified in \cite{BCM}. In this section we use these results to classify all such parasemifields, namely, we show the following Theorem \ref{classify}. 
Its proof is longish and will end only on page~\pageref{endproof}.

\begin{thm}\label{classify}
Let $S$ be an additively idempotent parasemifield, finitely generated as a semiring. Then $S$ is a (finite) product of 
parasemifields of the form
${G(T_i,v_i)}$, where $(T_i, v_i)$ are rooted trees and $G(T_i,v_i)$ are associated additively idempotent parasemifields (or equivalently $\ell$-groups), defined in Definition \ref{G(T,v)}.

Two such products $\prod_{i=1}^k G(T_i,v_i)$ and $\prod_{j=1}^{k^\prime} G(T_j^\prime,v_j^\prime)$ are isomorphic parasemifields if and only if
$k=k^\prime$ and there is some permutation $\sigma$ of $\{1, \dots, k\}$ such that for all $i$ we have $(T_i,v_i)\simeq(T_{\sigma(i)}^\prime,v_{\sigma(i)}^\prime)$ as rooted trees.
\end{thm}

The theorem can be viewed as happening in the category of additively idempotent parasemifields (in particular, it gives an equivalence between the subcategory of finitely generated objects and the subcategory consisting of finite products $\prod_{i=1}^k G(T_i,v_i)$).
Equivalently by Proposition \ref{3.1}, we can also view it in the isomorphic category of $\ell$-groups: 
It feels slightly more natural to define $G(T, v)$ there, in the language of $\ell$-groups.
First let us briefly introduce some notions related to rooted trees.

Note that a {\it rooted tree} $(T,v)$ is a (finite, non-oriented) connected graph $T$ containing no cycles and having a specified vertex, the {\it root} $v$.
By an {\it initial segment} $T^\prime$ of a rooted tree $(T, v)$ we shall mean a (possibly empty) subtree such that if $w\in T^\prime$, then all the vertices on the (unique) path in $T$ from $v$ to $w$ lie in $T^\prime$. If $T^\prime$ is a non-empty initial segment of a rooted tree $(T, v)$, the set of {\it next vertices} $N(T^\prime)$ is the set of all vertices $w\in T\backslash T^\prime$ such that there is $t\in T^\prime$ and an edge $(w, t)$ in $T$. If $T^\prime$ is empty, we set 
$N(T^\prime)=\{v\}$.
For a vertex $w$ define a tree $T_w\subset T$ consisting of exactly all the vertices $u\in T$ such that the (unique) path from $u$ to the root $v$ passes through $w$.

We are now ready to define $G(T, v)$.

\begin{definition}\label{G(T,v)}
Let $T$ be a tree with root $v$. Define an $\ell$-group $G(T, v)$ as follows: First attach a copy of the group of integers $\mathbb Z=\mathbb Z_w$ to each vertex  $w$ of $T$. 
Then $G(T, v)$ as an additive group is just the direct product of these groups $\mathbb Z_w$. We shall denote elements of $G(T, v)$ as tuples $(g_w)$ with $g_w\in\mathbb Z_w$.

Now take tuples $(g_w)$ and $(h_w)$ and define $(g_w)\vee (h_w)=(k_w)$ and $(g_w)\wedge (h_w)=(m_w)$ as follows: Let $T^\prime$ be the largest initial segment of $T$ such that $g_w=h_w$ for all $w\in T^\prime$. For $w\in T^\prime$, set $k_w=m_w=g_w(=h_w)$.
Take now $w\in N(T^\prime)$. Then $g_w\neq h_w$, without loss of generality assume that $g_w>h_w$. Then define $k_u=g_u$ and $m_u=h_u$ for all $u\in T_w$.

It is straightforward to check that $G(T, v)$ is indeed an abelian lattice ordered group; note that the lattice operations come essentially from some lexicographical ordering on $G(T, v)$ with respect to the structure of the tree.
\end{definition}

\smallskip

Let us note that the construction of $G(T, v)$ is closely related to the Hahn embedding: the tree $T$ is a chain if and only if the $\ell$-group $G$ is linearly ordered. In this case, the group $G(T, v)$ is exactly the group $\mathbb Z^n$ equipped with the lexicographic ordering, where $n$ is the number of vertices of $T$.

\ 

We will need a few properties of the construction of \cite{BCM} and of piecewise linear convex functions, especially in relation to being finitely generated as a semiring. They are collected in the following three lemmas.

\begin{lemma}\label{surject}
Let $\mathcal W=(W_0, W_1, \dots)$ be the stellar sequence corresponding to the $\ell$-group $G=\mathcal M([0,1]^n)/I$. Let $[0,1]^n\supset D_0\supset D_1\supset D_2\supset\dots$ be the corresponding sequence of complexes, let $D=\bigcap D_i$ and consider the $\ell$-group $\mathcal M(D)$ of restrictions of functions in $\mathcal M([0,1]^n)$ to $D$.

Then there is a surjection $G=\mathcal M([0,1]^n)/I\rightarrow\mathcal M(D)$. 
\end{lemma}

\begin{proof}
Let $\mathrm{res}: \mathcal M([0,1]^n)\rightarrow\mathcal M(D)$  be the restriction map and let $\pi:\mathcal M([0,1]^n)$ $\rightarrow\mathcal M([0,1]^n)/I$ be the projection. By the definition of $I$, if $\pi(f)=\pi(g)$ then $\mathrm{res}(f)=\mathrm{res}(g)$, and so $\mathrm{res}$ factors through $\pi$, i.e., 
$\mathrm{res}:\mathcal M([0,1]^n)$ $\rightarrow\mathcal M([0,1]^n)/I\rightarrow\mathcal M(D)$. Let $r:\mathcal M([0,1]^n)/I\rightarrow\mathcal M(D)$ be the corresponding map. Since $\mathrm{res}$ is a surjective homomorphism by definition, $r$ is surjective as well.
\end{proof}

\begin{lemma}\label{convex}
Let $A\subset [0,1]^n$ be a simplex and $f, g\in\mathcal W(A)$ convex functions. Then $\max(f, g)$ and $f+g$ are also convex.
\end{lemma}

\begin{proof}
A function $h$ is convex if the set $G(h)$ of all points above its graph is convex (in $A\times \mathbb R$), i.e., if the line segment between any two points in $G(h)$ lies in $G(h)$. Let $X, Y\in G(\max(f, g))$ and denote $XY$ the line segment between these points. Since $f$ and $g$ are both convex, $XY\in G(f)$ and $XY\in G(g)$. But then $XY\in G(f)\cap G(g)=G(\max(f, g))$. 

For $f+g$, choose $X=(x_1, x_2), Y=(y_1, y_2)\in G(f+g)$ ($x_1, y_1$ are $n$-tuples in $A$ and $x_2, y_2\in \mathbb R$). Then there are 
$X^\prime=(x_1,x^\prime), Y^\prime=(y_1,y^\prime)\in G(f)$ and $X^{\prime\prime}=(x_1,x^{\prime\prime}), Y^{\prime\prime}=(y_1,y^{\prime\prime})\in G(g)$ such that $x_2=x^\prime+x^{\prime\prime}$ and $y_2=y^\prime+y^{\prime\prime}$. If we now take points $X_0=(a, b)\in G(f)$ and $Y_0=(a,c)\in G(g)$ on the line segments 
$X^\prime Y^\prime$ and $X^{\prime\prime} Y^{\prime\prime}$, respectively, then the point $(a, b+c)$ is a point on the line segment $XY$ and lies in $G(f+g)$ (and each point of the line segment $XY$ is of this form).
\end{proof}

\begin{lemma}\label{sequence}
Let $a_1, a_2, \dots $ be a sequence of points in $D$ such that $\lim a_i=a\in [0,1]^n$. Then $\mathcal M(\{a_1, a_2, \dots\})$ and $\mathcal M(D)$ are not finitely generated semirings.
\end{lemma}

\begin{proof}
Assume that there are functions $f_1, \dots, f_k\in\mathcal M([0,1]^n)$ whose restrictions generate $\mathcal M(\{a_1, a_2, \dots\})$ as a semiring. 
Since each $f_i$ is piecewise linear, we can find a simplex $A$ such that each $f_i$ is linear on $A$ and infinitely many of the $a_i$ lie in $A$. Denote $B$ the set of all such $a_i$. Using the (surjective) restriction map $\mathcal M(\{a_1, a_2, \dots\})\rightarrow\mathcal M(B)$, we see that the functions $f_1, \dots, f_k$ generate  $\mathcal M(B)$ as well.

Now consider the subset $M$ of $\mathcal M(A)$ semiring-generated by $f_1, \dots, f_k$. Since each linear function is convex, each function in $M$ is convex by Lemma \ref{convex}. But there are clearly functions in $\mathcal M(B)$ which are not restrictions of convex functions on $A$, a contradiction.

The restriction map is a surjection from $\mathcal M(D)$ onto $\mathcal M(\{a_1, a_2, \dots\})$. Thus neither $\mathcal M(D)$ is finitely generated.
\end{proof}

Note that the same proof shows the following corollary:

\begin{corollary}\label{simplex}
Let $A\subset [0,1]^n$ be a simplex of dimension $\geq 1$. Then $\mathcal W(A)$ is not a finitely generated semiring.
\end{corollary}

We are now ready to start discussing the structure of additively idempotent parasemifields. We will first show that our parasemifield $S$ is a direct product of finitely many parasemifields corresponding to germs of functions at certain points.

\begin{definition}\label{germ}
Let $p$ be a point in $[0,1]^n$ and let $\mathcal P=([0,1]^n\supset P_0\supset P_1\supset\dots)$ be a sequence of complexes  
such that $\bigcap P_i=\{p\}$. Then we define the \emph{$\mathcal P$-germ of functions} at $p$ as  
$\mathcal M_{\mathcal P}(p)=\mathcal M([0,1]^n)/I$, where $I$ is the ideal corresponding to the sequence $\mathcal P$, i.e., $I$ consists of functions 
$f\in\mathcal M([0,1]^n)$ such that $f(P_i)=0$ for some $i$.
\end{definition}

\smallskip

The germ of functions at a point $p$ is exactly what it should intuitively be: it is the set of all functions viewed locally at $p$ ``in the directions given by $\mathcal P$".

\begin{proposition}\label{product}
Let $S$ be an additively idempotent parasemifield, finitely generated as a semiring. View $S$ as a (unital) $\ell$-group and let $\mathcal W=(W_0, W_1, \dots)$ be the corresponding stellar sequence, 
$\mathcal D=([0,1]^n\supset D_0\supset D_1\supset\dots)$ the corresponding sequence of complexes, and $I$ the defining ideal. 
Then $D=D(\mathcal W)=\bigcap D_i=\{d_1, \dots, d_k\}$ is finite and $S=\mathcal M([0,1]^n)/I$ is isomorphic to the 
direct product of $S_i=\mathcal M_{\mathcal D^i}(d_i)$, where $\mathcal D^i=([0,1]^n\supset D^i_0\supset D^i_1\supset\dots)$ with $D^i_j:=D_j\cap C^i$ for some fixed simplex $C^i$ containing an open neighborhood of the given point $d_i$.
\end{proposition}

\begin{remark}\label{imprecise}
The formulation of Proposition \ref{product} is fairly technical, but the idea is simple. The intersection $D$ is finite and the parasemifield $S$ will decompose as a direct product of parasemifields $S_i$, each of which corresponds to a germ of functions at a point $d_i\in D$.

Note that strictly speaking, the local sequences of complexes $\mathcal D^i$ we are using do not come from a stellar sequence. This is just a technicality, though: we can modify the stellar sequence $\mathcal W$ by first deleting all the simplexes outside of $C^i$ (using suitable subdivisions) and only then continuing with the stellar transformations which created $\mathcal W$. This produces a stellar sequence $\mathcal W^i$ whose corresponding sequence of complexes is $\mathcal D^{i\prime}=([0,1]^n\supset D^{\prime}_1\supset D^{\prime}_1\supset\dots\supset D^{\prime}_k \supset D^i_1\supset D^i_2\supset\dots)$, which differs from $\mathcal D^i$ only in finitely many complexes, and so produces the same germ of functions (as defined in \ref{germ} above).
\end{remark}

\begin{proof} (of Proposition \ref{product})
Assume that $D$ is not finite. Since $D\subset [0,1]^n$, we see that $D$ has a cummulation point. Thus by Lemma \ref{sequence} it follows that $\mathcal M(D)$ is not finitely generated. By Lemma \ref{surject}, 
$\mathcal M([0,1]^n)/I=S$ surjects onto $\mathcal M(D)$, and so neither $S$ is a finitely generated semiring.

Thus $D=\{d_1, \dots, d_k\}$ is finite and we can find suitable disjoint simplexes $C^i$ containing open neighborhoods of the points $d_i$ and define $\mathcal D^i$
and $S_i=\mathcal M_{\mathcal D^i}(d_i)$ as in the statement of the proposition. Then the restriction map gives a surjection $r: S\rightarrow\prod S_i$ similarly as in Lemma \ref{surject}. 

To show that $r$ is injective, assume that $r(f)=0$ for some $f\in \mathcal M([0,1]^n)$, i.e., there is $j$ such that $f(D^i_j)=0$ for all $i$. 
We want to show that $\pi(f)=0$. Since an open neighborhood of $D=\bigcap D_i$ is contained in $\bigcup_i D^i_j$, we see that there is $k$ such that $D_k\subset \bigcup_i D^i_j$.
Thus $f(D_k)=0$, which means that $f\in I$ and $\pi(f)=0$.
\end{proof}

Therefore to finish the classification we just need to describe the structure of the germs $\mathcal M_{\mathcal D}(d)$. This will be given in terms of $\ell$-groups $G(T, v)$ associated to rooted trees,  defined in Definition \ref{G(T,v)}.

\begin{proposition}
Let $\mathcal W=(W_0, W_1, \dots)$ be a stellar sequence, $\mathcal D=([0,1]^n\supset D_0\supset D_1\supset D_2\supset\dots)$ the corresponding sequence of complexes and $D=\bigcap D_i$. Assume that  $D=\{d\}$ has one element. Then the corresponding $\ell$-group of germs of functions $G=\mathcal M([0,1]^n)/I=\mathcal M_{\mathcal D}(d)$ is either not finitely generated as a semiring or is isomorphic to an $\ell$-group $G(T, v)$ associated to a (finite) rooted tree $(T,v)$.
\end{proposition}

\begin{proof}
Assume that $\mathcal M_{\mathcal D}(d)$ is finitely generated as a semiring.

To prove the proposition we shall modify the sequence $\mathcal D$ in several steps while preserving the $\ell$-group $\mathcal M_{\mathcal D}(d)$. The fairly long proof is divided into 5 steps:

\

{\it 1. Simplexes containing $d$}

First form a new sequence of complexes $\mathcal D^1=(D_0^1\supset D_1^1\supset\dots)$, where $D^1_i$ is obtained from $D_i$ by recursively removing all maximal
simplexes not containing $d$. Note that the simplexes not containing $d$ play no role in determining the germ of local functions, and so $\mathcal M_{\mathcal D^1}(d)=\mathcal M_{\mathcal D}(d)$. 
Also note that $\mathcal D^1$ is still obtained from a stellar sequence (taking into account the potential need for making modifications as in Remark \ref{imprecise}; we shall not mention this in the future).

\

{\it 2. Stable subspaces}

By a {\it stable line} in $\mathcal D^1$ we shall mean a line $\ell$ passing through $d$ such that $\ell\cap D^1_i$ is a line segment (and not just the point $d$) for each $i$. This means that while the stellar transformations which give $\mathcal D^1$ may (and will) subdivide the 1-dimensional simplex which gives a line segment lying on $\ell$, they will never delete this simplex.

Let us point out that stable lines give non-trivial elements in $\mathcal M_{\mathcal D}(d)$: the germ of linear functions on $\ell$ will lie in $\mathcal M_{\mathcal D}(d)$. A linear function on a line is determined by its slope (and value at the point $d$) and since the functions we are considering are restrictions of linear functions with integral coefficients, the set of possible slopes is $\mathbb Z$. Thus to each stable line $\ell$ corresponds a copy of $\mathbb Z\subset \mathcal M_{\mathcal D}(d)$.

Similarly for $k\geq 1$ we can define a {\it stable $k$-subspace} in $\mathcal D^1$ as a $k$-dimensional (affine) space $L$ containing $d$ such that $L\cap D^1_i$ has dimension $k$ for each $i$. (A stable 1-subspace is just a stable line.)

By the definition of stable subspaces it follows that if a simplex in $D^1_i$ intersects every stable line only in the point $d$ (and thus the same is true for the intersection with any stable subspace), then it does not contribute to $\mathcal M_{\mathcal D}(d)$. Therefore we can form $D^2_i$ and $\mathcal D^2$ by omitting all such simplexes with no non-trivial intersection with a stable line. Then $\mathcal M_{\mathcal D^2}(d)=\mathcal M_{\mathcal D}(d)$.

\

{\it 3. Simplexes defined using the generators}

By an open simplex we shall mean a point, or the interior of a $k$-simplex for $k\geq 1$.

Denote the (semiring) generators of $S=\mathcal M_{\mathcal D}(d)$ by $f_1, \dots, f_k$ (as usual, we identify a function $f\in\mathcal M([0,1]^n)$ with its
image in $\mathcal M_{\mathcal D}(d)$).  Each of these functions is piecewise linear, and so there is a finite set $\mathcal P_0$ of open simplexes which
cover $[0,1]^n$ and such that the restriction of each $f_j$ to any $P\in \mathcal P_0$ is linear. In fact, we can modify $\mathcal P_0$ to get the following
lemma:

\begin{lemma}\label{mathcal P}
There is a finite set $\mathcal P$ of open simplexes such that

(i) elements of $\mathcal P$ are pairwise disjoint,

(ii) the restriction $f_j|P$ is linear for all $j$ and all $P\in\mathcal P$,

(iii) $\dim P\cap D_i^2=\dim P$ for all $i$ and all $P\in\mathcal P$,

(iv) $P\cap \ell=\emptyset$ for all stable lines $\ell$ and all $P\in\mathcal P$ with $\dim P>1$,

(v) for each $e$ and each $P\in\mathcal P$ with $\dim P>e$ there is exactly one $Q\in\mathcal P$ such that $\dim Q=e$ and $Q\subset \bar{P}$ 
($\bar{P}$ denotes the closure in $\mathbb R^n$),

(vi) $\mathcal M_{\mathcal D^3}(d)=\mathcal M_{\mathcal D}(d)$, where
$\mathcal D^3=\mathcal D^2\cap U =(D^2_0\cap U\supset D^2_1\cap U\supset\dots)$
and $U=\bigcup_{P\in\mathcal P} P$.
\end{lemma}

Note that the simplexes  $P\in\mathcal P$ from the lemma can be viewed as a refinement of the notion of stable subspaces.

\begin{proof}
We shall modify $\mathcal P_0$ recursively in several steps while making sure that $\mathcal M_{\mathcal D^2\cap U}(d)$ remains unchanged and 
equal to $S=\mathcal M_{\mathcal D}(d)$ (this is clearly true at the beginning as $\bigcup_{P\in\mathcal P_0} P=[0,1]^n$).

To start, let $\mathcal P=\mathcal P_0$. Now recursively repeat the following set of modifications:

1. For determining $S$ are relevant only those simplexes $P\in\mathcal P$ which have non-empty intersection with infinitely many (and hence all) of the $D_i^2$. 
Hence we can delete all other $P$ from $\mathcal P$. Continue to Step 2.

2. If there is $P\in\mathcal P$ and finitely many open simplexes $S_1, \dots, S_a\subset P$ such that $\dim S_i<\dim P$ and 
$P\backslash(S_1\cup\dots\cup S_a)$ has non-empty intersection with all $D^2_i$, then replace $P$ by $S_1, \dots, S_a$ in $\mathcal P$.  
Return to Step 1 if $\mathcal P$ has been modified, 
else continue to Step 3. 

Note that $\mathcal P$ has finitely many elements at any time and dimensions of elements of $\mathcal P$ are decreasing, so this step will happen only finitely many times. Also note that after being done with steps 1 and 2, $\mathcal P$ contains only open simplexes $P$ with $\dim(P\cap D_i^2)=\dim P$ for all $i$.

3. Assume that $d\neq P\in\mathcal P$ has non-empty intersection with infinitely many stable lines. Arguing as in the proof of Lemma \ref{sequence}
we see that $S$ is then not finitely generated as a semiring, a contradiction: Namely, any function $f$ which is linear along finitely many of these lines 
(and suitably defined at the other lines) will be non-trivial in $S$. By considering the set of slopes of $f$ along these lines, it's easy to construct
a function $f\in S$ which will not be convex on $P$. But this contradicts Lemma \ref{convex} as all the semiring generators are linear on $P$.
Thus every $d\neq P\in\mathcal P$ has non-empty intersection with only finitely many stable lines.

Suppose that $\dim P>1$, $\ell$ is a stable line and $\ell\cap P\neq\emptyset$. Choose then a $(\dim P-1)$-dimensional hypersurface $H$ containing $\ell$ and subdivide $P$ along this hypersurface, i.e., $P=(P\cap H)\cup (P\backslash H)$ and $P\backslash H$ has two connected components, $P_1$ and $P_2$. We can choose $H$ so that 
$P\cap H, P_1$, and $P_2$ are all open simplexes;
in $\mathcal P$ then replace $P$ by $P\cap H$, $P_1$ and $P_2$.

After doing this finitely many times (because for each $P$ there are only finitely many stable lines), we arrive at $\mathcal P$ satisfying property (iv). Return to Step 1 if $\mathcal P$ has been modified, 
else continue to Step 4. 

4. Assume that there are $P, Q, R\in\mathcal P$ such that $Q, R\subset \bar{P}$ and $Q\not\subset \bar R$ and $R\not\subset \bar Q$. 
Take such a $P$ of the smallest dimension. Since $Q$ and $R$ are disjoint, 
we can again subdivide $P$ by a $(\dim P-1)$-dimensional hypersurface $H$ as above so that ($Q\subset \bar{P_1}$ and $R\subset \bar{P_2}$)
or ($Q\subset \bar{P_1}$ and $R\subset \bar{P\cap H}$) or ($R\subset \bar{P_1}$ and $Q\subset \bar{P\cap H}$)
(and replace $P$ in $\mathcal P$ by $P\cap H$, $P_1$ and $P_2$).

After doing this finitely many times (because for each $P$ such a situation can occur only finitely many times), we arrive at $\mathcal P$ satisfying property (v): We have just ensured the uniqueness of such $Q$, its existence easily follows from the fact that $\dim P\cap D_i^2=\dim P$.

Return to Step 1 if $\mathcal P$ has been modified, else we are done. 

\

Note that the whole algorithm terminates after finitely many steps and that (i) -- (vi) are satisfied at the end, completing the proof of Lemma \ref{mathcal P}.
\end{proof}

\

{\it 4. Construction of the tree $T$}

Now we can easily construct a rooted tree $(T, v)$ attached to the sequence $\mathcal M_{D^3}(d)$ obtained using Lemma \ref{mathcal P}:
Associate a vertex $v_P$ to each $P\in\mathcal P$, there will be an edge connecting vertices $v_P$ and $v_Q$ if and only if
($P\subset \bar Q$ and $\dim P=\dim Q-1$) or ($Q\subset \bar P$ and $\dim Q=\dim P-1$). The vertex $v_d$ is the root $v$.

By Lemma \ref{mathcal P} we see that $(T, v)$ is a (connected) rooted tree.

\

{\it 5. Description of $M_{\mathcal D}(d)$}

The germ of a function $f$ in $\mathcal M_{\mathcal D^3}(d)=M_{\mathcal D}(d)$ can have any value at $d$, which gives the $\mathbb Z_v$ at the root $v$ of the tree $T$.

Given $f\in \mathcal M_{\mathcal D^3}(d)$, choose a small ball $B$ containing $d$ so that the restriction of $f$ to $B\cap r$ is linear for all rays $r\ni d$. Since we are considering
only the germ of functions at $d$, $f$ is uniquely determined by $f|B$ as an element of $\mathcal M_{\mathcal D^3}(d)=M_{\mathcal D}(d)$.

Take a 1-simplex $P\in\mathcal P$. The value of $f$ at the endpoint $d$ of $\bar P$ has already been selected, and so the restriction $f|(P\cap B)$
(which is linear by Lemma \ref{mathcal P}, (ii)) is uniquely 
determined by its value at any point $p\in P\cap B$. The choice of this value gives the $\mathbb Z_{v_P}$ at the vertex $v_P$ of the tree $T$.

After having dealt with all the 1-simplexes, take a 2-simplex $P\in\mathcal P$. There is a unique 1-simplex $Q\in \mathcal P$, $Q\subset \bar P$; 
$f|(Q\cap B)$ has already been determined, and so the restriction $f|(P\cap B)$ is uniquely 
determined by its value at any point $p\in P$. The choice of this value gives the $\mathbb Z_{v_P}$ at the vertex $v_P$ of the tree $T$.

We can continue in this way, successively dealing with simplexes of larger and larger dimensions, until we have covered the whole tree $T$ and 
uniquely determined the function $f|B$ as an element of $\mathcal M_{\mathcal D^3}(d)=M_{\mathcal D}(d)$.

Now it is straightforward to check that the $\ell$-group $M_{\mathcal D}(d)$ is exactly $G(T, v)$.
\end{proof}

Together with Proposition \ref{product}, this finishes the proof of Theorem \ref{classify}, except for the uniqueness part. 
This follows from the proof and from the uniqueness statement of \cite{BCM}, Corollary 5.4.
However, for the sake of completeness and to make sure that our classification is indeed independent of the choice of order-unit in Theorem \ref{unital}, let us give (a sketch of) a direct proof.

Assume that $G=\prod_{i=1}^k G(T_i, v_i)$ is one of our $\ell$-groups from Theorem \ref{classify}, given abstractly as an $\ell$-group $(G, +, -, 0, \vee, \wedge)$, i.e., without specifying the corresponding rooted tree structure (or the order-unit). To make the notation more uniform, we consider the disjoint union $F$ of the rooted trees $(T_i, v_i)$ as a ``rooted forest" 
$\mathcal F=(F, v_1, v_2, \dots, v_k)$.

We shall show how to reconstruct this rooted forest $\mathcal F$ from $G$, which will then imply the uniqueness statement of the theorem. 

First let's introduce some notation for the ``standard" basis" of $\prod_{i=1}^k G(T_i, v_i)$. 
In Definition \ref{G(T,v)} we have attached a copy $\mathbb Z_w$ of the additive group of integers to each vertex $w$. 
Denote by $b(w)$ the element of $G$ which corresponds to $1\in\mathbb Z_w$, i.e., $b(w)=(g_v)_{v\in F}$ is the tuple with $g_w=1$  and $g_v=0$ if $v\neq w$. Note that by definition we have $b(w)>0$.
We shall say that an element $g\in G$ is \emph{infinitesimally smaller} than $h\in G$ if $ng<h$ for all $n\in\mathbb Z$ and denote this by $g\ll h$. We say that an element $g\in G$ is \emph{infinitesimal} if $g\ll h$ for some $h\in G$.

Let us now try to identify the basis elements $b(v)$ corresponding to roots $v=v_i$.
Define $B_0=\{g_1, \dots, g_m\}$ as a maximal set of elements which satisfy all of the following properties for all pairs $i\neq j$:
\begin{itemize}
\item $g_i>0$
\item $g_i$ is not infinitesimal
\item $g_i$ is not a sum of positive non-infinitesimal elements
\item $g_i\vee g_j=g_i+g_j$
\end{itemize}

Considering the elements $g_i$ as linear combinations of the basis $b(w)$, it is easy to see that $k=m$ and that there are infinitesimal elements $h_i\ll g_i$ such that 
$\{g_1+h_1, \dots, g_k+h_k\}=\{b(v_1), \dots, b(v_k)\}$. After permuting the indices if necessary, we can assume that $g_i+h_i=b(v_i)$.
Hence upto the infinitesimal elements $h_i$, we see that $B_0$ is the set of basis elements corresponding to the roots $v_i$.

Now for $i=1, \dots, k$ define $G_i=\{g\in G\mid g\ll g_i\}$. This is an $\ell$-subgroup of $G$ isomorphic to the $\ell$-group $G(\mathcal F_i)$ attached to a rooted forest $\mathcal F_i$, which is obtained by removing 
the root $v_i$ from the tree $T_i$ and designating the vertices $v\in N(v_i)$ (i.e., those that are connected to $v_i$ by an edge in $T_i$) as the roots of the trees in the forest $\mathcal F_i$.

We can now proceed in the same way with each $G_i$ and define a set $B_i$ of elements that correspond to basis elements $b(v)$, $v\in N(v_i)$ (again upto elements that are infinitesimally smaller).

Proceeding by induction in this fashion, we eventually define an element $g_w$ for each $w\in F$
so that the set $\{g_w\mid w\in F\}$ with the ordering $\ll$ is isomorphic to the rooted forest $\mathcal F$ (viewed as an ordered set whose maximal elements are the roots).
Note that we have defined the ordered set $\{g_w\mid w\in F\}$ intrinsically, without referring to the forest $\mathcal F$ (or the chosen order-unit).  

Assume now that $\prod_{i=1}^k G(T_i,v_i)$ and $\prod_{j=1}^{k^\prime} G(T_j^\prime,v_j^\prime)$ are isomorphic $\ell$-groups. As above, we can attach to them rooted forests 
$\mathcal F$ and $\mathcal F^\prime$, respectively, which then have to be isomorphic rooted forests. This proves the uniqueness statement of Theorem \ref{classify}.
\hfill\qed

\label{endproof}

\

Let us note that as a group, each $G(T_i, v_i)$ is just $\mathbb Z^{n_i}$ for some $n_i$, and so we obtain the following corollary to Theorem \ref{classify}:

\begin{corollary}
If an additively idempotent parasemifield is finitely generated as a semiring $S(+, \cdot)$, then it is
finitely generated as a group $S(\cdot)\simeq\mathbb Z^n$.
\end{corollary}

We are not aware of any more direct or elementary  proof of this surprising fact. It would certainly be very interesting to obtain one.

\def\bibname{Bibliography}

\end{document}